\documentclass[12pt]{amsart}
\usepackage{amsmath, amsthm, amscd, amsfonts}
\usepackage{graphicx}

\newfont{\aj}{eufm10 at12pt}
\newfont{\ajk}{eufm10 at10pt}
\theoremstyle{plain}

\newtheorem{them}{Theorem}[section]
 
 \newtheorem{lem}[them]{Lemma}
 
 \newtheorem{exam}[them]{Example}
 \theoremstyle{definition}
 \newtheorem{df}[them]{Definition}
 \theoremstyle{remark}
 \newtheorem{rem}[them]{Remark}
 
 \numberwithin{equation}{section}

\begin{document}

%------------------------------------------------------------------------------------%
%%Don not change any thing in this part
%\leftline{ \scriptsize \it Bulletin of the Iranian Mathematical
%Society  Vol. {\bf\rm XX} No. X {\rm(}201X{\rm)}, pp XX-XX.}

\vspace{1.3 cm}

%------------------------------------------------------------------------------------%
\title{On the notion of fuzzy shadowing property}
\author{Mehdi Fatehi Nia
\\
\small{ Department of Mathematics, Yazd University, 89195-741 Yazd, Iran.}
\\ \small{ Tel:+98 351 38122716}
\\
\small{E-mail:fatehiniam@yazd.ac.ir}}

\thanks{{\scriptsize
\hskip -0.4 true cm MSC(2010): Primary: 37C50; Secondary: 37C15,
11Y50.
\newline Keywords: Fuzzy metric, Fuzzy discrete dynamical systems, Fuzzy shadowing, Fuzzy ergodic shadowing, Fuzzy topological mixing.\\
%Received: 30 April 2009, Accepted: 21 June 2010.\\
%$*$Corresponding author
%\newline\indent{\scriptsize $\copyright$ 2011 Iranian Mathematical
%Society}
}
}

\maketitle
%------------------------------------------------------------------------------------%
%This part will be filled in by BIMS
%\begin{center}
%Communicated by\;
%\end{center}
%------------------------------------------------------------------------------------%
\begin{abstract}
This paper is concerned with the study of fuzzy dynamical systems. Let $(X,M,\ast)$ be a fuzzy metric space in the sense of George and Veeramani. A fuzzy discrete dynamical system is given by any fuzzy continuous self-map defined on $X$. We introduce the various fuzzy shadowing and fuzzy topological transitivity on a fuzzy discrete dynamical systems. Some relations between this notions have been proved.
\end{abstract}

\vskip 0.2 true cm

%------------------------------------------------------------------------------------%

\pagestyle{myheadings}
\markboth{\rightline {\scriptsize  MEHDI FATEHI NIA}}
         {\leftline{\scriptsize On the notion of fuzzy shadowing property}}

\bigskip
\bigskip

%------------------------------------------------------------------------------------%
%------------------------------------------------------------------------------------%
%\renewcommand{\lineskip}{1.5}

\section{Introduction} %In this paper  we will present an observational consideration of the fuzzy metric spaces. We will
%define the notion of fuzzy attractor sets as the basic objects for
 %relative semi-dynamical systems.  We will present a
%discussion on fuzzy attractor sets in standard fuzzy metric
%spaces. Then we will prove basic theorems for fuzzy attractor
%sets.\\
\emph{Fuzzy topology }is an extension of the ordinary topology which has become an area of active research in the recent
years, because of its numerous applications \cite{S}.
%One of the most important problems in fuzzy topology is to obtain the concept of fuzzy metric space. This problem has been
 %investigated by many researchers from different points
%of view. In particular,
George and Veeramani (see \cite{Na,Nb}) modified and studied a notion of fuzzy metric $M$ on a set $X$ via of continuous $t-$norms  which introduced by Kramosil and Michalek \cite{KM}.  From now on, when we talk about fuzzy metrics we refer to this type of fuzzy metric spaces. \\George and
Veeramani proved that $M$ induces a topology on $X$. This topology is not the same as the fuzzy topology. Actually, this topology also can be constructed on each fuzzy metric space in the sense of Kramosil and Michalek \cite{GR,KM,M}. We should remark that the topology on $X$ deduced from a fuzzy
metric $M$, does not depend on the t-norm \cite{Nb}.\\
The shadowing property and chain mixing are two of the most important concepts in discrete dynamical systems \cite{B}, which are closely
related to stability and chaotic behavior of dynamical systems, see, for instance \cite{R,W,YA,YB} and they are essential parts of stability and ergodic theory \cite{B}. In this direction the concept of ergodic shadowing property, topological mixing and chain mixing are  another formal
ways to get some special dynamical properties, for discrete dynamical system $(X, f)$, where $f$ is an onto continuous map (see \cite{CL, CKY, KP, RW, KS,W,YA}).\\
 The study of discrete fuzzy dynamical systems has been done by many authors and different
properties of fuzzy discrete dynamical systems were considered. For example, study of some chaotic properties of the discrete fuzzy dynamical system  has been done by some authors  \cite{Ka,Kc,Kb,RC}.
\\
Nearness depends on the time and we have not introduced any distance as a metric. In fact, fuzzy metric determines the rate of nearness without introducing a distance number, this is the main idea of the fuzzy metric spaces. This theory implies that \emph{the rate of nearness is more important than the notion of distance.} For this reason we use of the word fuzzy, since there is not any distance in fuzzy metric space theory. In this paper we have extend the notion of shadowing by using of the rate of nearness instead of distance. In coding theory distance is not reachable, but we can determine the rate of nearness. Since shadowing is one of the important property in the theory of discrete dynamical
systems as it is close to the stability of the system and also to the chaotic
behavior of the systems \cite{R,W,YA,YB}, this extension can be a useful tool to study stability and chaos in fuzzy discrete dynamical systems.
\\
The present paper is organized as follows: In Section $2$, several basic definitions and notations are presented.
In Section $3$, we consider the concept of fuzzy discrete dynamical systems which is a generalization of discrete dynamical systems \cite{K,Z}. We use \textbf{$F-$....} instead of Fuzzy....\\
The notions of $\delta-F$-pseudo-orbit, $F-$ shadowing property,
$F-$topological mixing, $F-$chain transitive, $F$-ergodic shadowing property, have been
developed and presented. Some propositions
on the above notions have also been established. Actually, we show that the following statements hold:\\
$1)$ Any mapping with $F$-ergodic shadowing property is $F-$chain transitive.\\
$2)$ If $f$ has the $F-$ergodic shadowing property then for any natural number $k$ the mapping $f^{k}$ has also the $F-$ergodic shadowing property.\\
$3)$ $F-$ergodic shadowing property implies $F-$shadowing property.\\
$4)$ If $f$ has the $F$-shadowing and $F$-chain mixing properties then it has the $F-$topologically mixing property.\\In Section $4$ we have introduced few examples to illustrate the definitions above.
A tent map is one of the most popular and the simplest chaotic maps \cite{CKY}. In Example \ref{e6} and Remark \ref{rem2}, we consider the family of tent maps and show that they have the\\ $F$-shadowing and $F$-chain mixing properties. Also, in Example \ref{e7} we show that the classical shadowing property is a special case of fuzzy
shadowing property and they are essentially different.\\ As usual
the paper finishes  some concluding remarks and open problems (Section $5$).
%for which $M$ is a fuzzy metric(see Definition 1).
\section{Preliminaries}This section contains two major subsections: Fuzzy Metric Space and Discrete Dynamical Systems. We start this section by recalling some pertinent concepts.\\
\textbf{Fuzzy Metric Space}
 \\ \indent Let $\ast $ be a continuous \emph{t-norm}, i.e., it is a binary operation from \\$[0,1]\times
[0,1]$ to $[0,1]$ with the following conditions:\\
$i)  ~\ast $ is
associative and commutative;\\ $ii)~ \ast$ is continuous;\\
$iii) ~a\ast1=a$ for all $a \in [0,1];$ \\ $iv)~a \ast b \leq c
\ast d$ whenever $a \leq c$ , $b \leq d$ and~ $a,b,c,d \in [0,1].$
\begin{lem}\label{l1} \cite{Na}. For any $r_{1}>r_{2}$, we can find $r_{3}$ such that\\ $r_{1}\ast r_{3}\geq r_{2}$ and for any $r_{4}$ we can find $r_{5}$ such that \\$r_{5}\ast r_{5}\geq r_{4}$, $(r_{1},r_{2},r_{3},r_{4},r_{5}\in (0,1)).$\end{lem}
 We also assume that $(X,~M,~\ast)$ is a \emph{fuzzy metric space} \cite{Na,Nb} i.e., $X$ is a nonempty set, $\ast$ is a
continuous t-norm and\\ $M:X \times X \times
(0,\infty)\longrightarrow [0,1]$ is a mapping with the
following properties:\\ \indent For every $x,y,z \in X$ and $t,s > 0$\\
$1)~M(x,y,t)>0$;\\ $2)~M(x,y,t)=1$ if and only if $x=y;$ \\
$3)~M(x,y,t)=M(y,x,t); $ \\ $4)~M(x,z,t+s)\geq M(x,y,t)\ast
M(y,z,s);$ \\ $5)~M(x,y,.):(0,\infty)\longrightarrow \textbf{[}0,~1]$ is a
continuous map.\\
 If $(X,M,\ast)$ is a fuzzy metric space, we will say that $M$ is
a fuzzy metric on $X.$ The function $M(x,y,t)$ denotes the degree of nearness between $x$ and $y$ respect to $t$. Many interesting examples of fuzzy metric spaces
can be found in \cite{GMS}.
\begin{lem} \cite{MG}. $M(x,y,.):(0,\infty)\longrightarrow (0,1]$ is nondecreasing for all $x,y$ in $X.$
 \end{lem}
\begin{df} \cite{Na}Let
$(X,M,\ast)$ be a fuzzy metric space. A set \\$A\subset X$ is
called a \emph{fuzzy open set} if for any $x\in A$ there exists $0<r<1$ and
$T_{0}\in (0,\infty)$ such that if $M(x, y, t)>1-r$, for all $t>T_{0}$ then $y\in A$.
 \end{df} Since $ M(x, y, t)\geq M(x, y,
T_{0})$ for all $t>T_{0},$ then in the above definition it is
sufficient that for one $T_{0}\in (0,~\infty)$ if \\$M(x, y,
T_{0})>1-r$ then $y\in A.$
\begin{df} \cite{Nb}. Let $(X, M, \ast)$ be a
fuzzy metric space. For $t>0$, the \emph{open ball} $B(x,r,t)$ with center
$x\in X$ and radius \\$0<r<1$ is defined by: $$B(x, r, t)= \{y \in X;
M(x, y, t)> 1-r\}.$$
Similarly, for $t>0$, the\emph{ closed ball} with the center $x \in X$ and
radius $ 0<r<1$ is defined by: $$B[x, r, t]= \{y \in X; M(x, y, t)\geq
1-r\}.$$
 \end{df}
 In \cite{MG} it has been proved that every fuzzy metric $M$ on $X$ generates a
Hausdorff first countable topology $\tau_{M}$ on $X$ which has a base the family of open sets of the form $$\{B(x, r, t): x\in X, r\in (0,1),~ t>0\}.$$
Functions M(x,y,t) presented in Examples from \ref{e1} to \ref{e3} will be utilized in the Section 4.\begin{exam} \cite{Nb} \label{e1}Let $(X, d)$ be a
metric space. Denote by $a.b$ the usual multiplication for all $a, b\in[0,1]$, and let $M_{d}$ be the function defined on $X\times X\times(0,\infty)$ by $M_{d}(x, y, t)= \frac{t}{t+ d(x, y)}$. Then $(X, M_{d}, ~.)$ is a fuzzy metric space called \emph{standard fuzzy metric space}, and $(M_{d}, ~.)$ will be called the standard fuzzy metric of $d$.
And the topology $\tau_{M_{d}}$ generated by $d$ coincides with the
topology  $\tau_{M_{d}}$ generated by the fuzzy metric $M_{d}$. \end{exam}
\begin{exam}\cite{GMS} \label{e2} Let $X=(0,1]$ and let $\varphi:\mathbb{R}^{+}\rightarrow (0,1]$ be a function given by $\varphi(t)=t$ if $t\leq 1$ and $\varphi(t)=1$ elsewhere. Define the function $M$ on $X\times X\times(0,\infty)$ by: \[ M(x, y, t) = \left\lbrace
  \begin{array}{c l}
  1& \text{if ~$~ x=y$},\\
    \frac{min\{x, y\}}{max\{x, y\}}.\varphi(t) & \text{if ~$~x\neq y$}.
  \end{array}
\right. \]
It is easy to verify that $(X, M,~.)$ is a fuzzy metric on $X$, where $a.b=ab$ for all $a, b\in X$.\\
Notice that $B(x,\frac{1}{2}, \frac{1}{2})= {x}$ for each $x\in ~X$
and so $\tau_{M}$ is the discrete topology.\end{exam}
\begin{exam}\textbf{\cite{GMS}} \label{e3} Let $X=(0,~1]$ and define the function $M_{1}$ on \\$X\times X\times(0, \infty)$ by: \[ M_{1}(x, y, t) = \left\lbrace
  \begin{array}{c l}
  1& \text{if ~$~ x=y$},\\
    \frac{min\{x,y\}}{max\{x, y\}} & \text{if ~$~x\neq y$}.
  \end{array}
\right. \]
 It is easy to verify that $(X, M_{1},~.)$ is a fuzzy metric on $X$, where \\$a.b=ab$ for all $a,b\in X$. Let $\epsilon>0$ and $a\in X$.
This is clear that \\$B(a, \epsilon, t)\subset B_{\epsilon}(a)$ and $B_{\epsilon}(a\epsilon)\subset B(a, \epsilon, t)$,
for all $t\in  \mathbb{R}^{+}$. So $\tau_{M_{1}}$ is the usual topology of $\mathbb{R}$.\end{exam}
\begin{them} \cite{Nb}. Let $(X, M, \ast)$ be a fuzzy metric space and $\tau_{M}$ be the topology
induced by the fuzzy metric. Then for a sequence $\{x_{n}\}_{n=1}^{\infty}$ in
$X,~lim_{n\rightarrow \infty}x_{n}= x$ if and only if $lim_{n\rightarrow \infty}M(x_{n}, x, t)=1$ for all $t>0$.  \end{them}
\indent A sequence $\{x_{n}\}_{n=1}^{\infty}$ in a fuzzy metric space $(X, M, \ast)$ is a Cauchy sequence if and
only if for each $\epsilon, t>0$ there exists $n_{0} \in \mathbb{N}$ such
that $M(x_{n}, x_{m}, t)>1-\epsilon$ for all $n, m\geq
n_{0}.$ A fuzzy metric space in which every
Cauchy sequence is convergent is called a complete fuzzy metric
space.
\begin{df} \textbf{\cite{Na}}Let $(X, M, \ast)$ be a fuzzy metric space
and \\$f :X \longrightarrow ~X$  be a fuzzy map. Then $f$ is said to
be a \emph{fuzzy continuous} map at the point $x_{0}$, if for any
$\epsilon\in(0,1)$ and $t>0$ there exists $\delta\in(0,1)$ and \\$t^{'}>0$ such that for each
$x$ with $M(x, x_{0}, t^{'})>1-\delta$. So\\
$M(f(x), f(x_{0}), t)>1-\epsilon$. \end{df}  $f$ is said to be a
fuzzy
continuous map if $f$ is fuzzy continuous at any point of $X$.
% We say that
%$f$ is called a fuzzy homeomorphism if $f$ is invertible and $f$
%and $f^{-1}$ are fuzzy continuous maps. Let $x$ be a point and let $f$ be a map.
%The orbit of $x$ under $f$ is the set of points $\{x, f(x), f^{2}(x), . . .\}$ and the $n-$th iteration $f^{n}$  of the map $f$ is defined inductively by $f^{1} = f$
% and $f^{n}= f o f ^{n-1}$ for $ n \geq2$.
\subsection*{Discrete Dynamical Systems}
\indent Here, we recall some notions and results of discrete dynamical systems which we require in this paper \cite{CKY,FG}.\\ \indent Let $(X, d)$ be a compact metric space and $f: X \longrightarrow X$ be a continuous map. For any two open subsets $U$ and $V$ of $X$, denote
\\$N(U, V , f ) = \{m \in \mathbb{N}; f^{m}(U) \cap V \neq \emptyset \}.$
When there is no ambiguity, we denote it by $N(U, V )$. We say that $f$ is \emph{topologically transitive} if for every pair of open sets
$U$ and $V$ in $X$, $N(U, V ) \neq \emptyset$. \emph{Topological mixing} means that for
any two open subsets $U$ and $V$, the set $N(U, V )$ contains any natural number $n \geq n_{0}$, for some fixed $n_{0}\in \mathbb{N}.$\\ \indent
Given a number $\delta > 0$, a $\delta$\emph{-pseudo orbit } for $f$ is a sequence $\{x_{i}\}_{1\leq i\leq b }$ such that  $~d( f (x_{i} ), x_{i+1}) < \delta$, for every $1 \leq i \leq b$. If $b <\infty$, then we say that the finite $\delta$-pseudo orbit $\{x_{i}\}_{1\leq i\leq b}$ of $f$ is a \emph{$\delta$-chain} of $f$ from $x_{1}$ to $x_{b}$ of length $b$. A point $x\in X$ is called a \emph{chain recurrent point} of $f$ if for every $\delta >0$, there is a $\delta$-chain from $x$ to $x$. Denote by $CR( f )$ the set of all chain recurrent points of $f$. A sequence $\{x_{i}\}_{1\leq i\leq b}$ is said to be \emph{$\epsilon$-shadowed} by a point $x$ in $X$ if $d( f^{i} (x), x_{i}) < \epsilon$ for each $1 \leq i\leq b$. A mapping $f$ is said to have \emph{shadowing property} if for any $\delta> 0$, there is a $\epsilon >0$ such that every $\delta$-pseudo orbit of $f$ can be $\epsilon$-shadowed by some point in $X$.\\
\indent A mapping $f$ is called \emph{chain transitive} if for any two points \\$x, y \in X$ and any $\delta >0$ there exists a $\delta$-chain from $x$ to $y$. The
mapping $f$ is called chain mixing if for any two points $x, y \in X$ and any $\delta >0$, there is a positive integer $n_{0}$ such that for any
integer $n\geq n_{0}$ there is a $\delta$-chain from $x$ to $y$ of length $n$.\\
 Given a sequence $\eta = \{x_{i}\}_{i\in \mathbb{N}}$, denote $$Npo(\eta, \delta) = \{i|d( f (x_{i}), x_{i+1})\geq \delta\}$$ and $$Npo_{n}(\eta, \delta) = Npo(\eta, \delta) \cap \{1, . . . ,n - 1\}.$$ For a
sequence $\eta$ and a point $x$ of $X$, denote $$Ns(\eta, x, \delta) = \{i|d( f ^{i}(x), x_{i})\geq \delta\}$$ and $$Ns_{n}(\eta, x, \delta) = Ns(\eta, x, \delta) \cap\{1, . . . , n-1\}.$$
\begin{df} \cite{FG}. A sequence
$\eta$ is a \emph{$\delta$-ergodic pseudo orbit} if  $Npo(\eta, \delta)$ has density zero, it
means that $lim_{n\rightarrow \infty}\frac{card(Npo_{n}(\eta,~ \delta))}{n}= 0$,
where $card(.)$ denotes the cardinal number. A $\delta$-ergodic pseudo orbit $\eta$ is said to be \emph{$\epsilon$-ergodic shadowed} by a point $x$ in $X$ if
\\$lim_{n\rightarrow \infty}
\frac{card(Ns_{n}(\eta, x, \epsilon))}{n}= 0$.\end{df}
In other words, $Ns(\eta, x, \epsilon)$, has density zero. A mapping
$f$ has ergodic shadowing property if for any $\epsilon > 0 $ there is a $\delta >0$ such that any $\delta$-ergodic pseudo orbit of f can be $\epsilon$-ergodic
shadowed by some point in $X$.
\\The shadowing property and chain recurrence is one of the most important concepts in dynamical systems \cite{B}, which are closely
related to stability and chaos of systems, see, for instance \cite{R,W,YA} and are essential parts of stability and ergodic theory \cite{B}.
 \\In \cite{FG} Fakhari and Ghane  prove the following theorem:
\begin{them}\textbf{\cite{FG}}\label{gha} Let $f$ be a continuous onto map of a compact metric space $X$. For the dynamical system $(X, f )$, the following properties
are equivalent:\\
(a) ergodic shadowing,\\
(b) shadowing and chain mixing,\\
(c) shadowing and topological mixing.
%(d) pseudo-orbital specification.
We are going to study other concepts and their relations in another papers.
\end{them}
\section{Fuzzy Dynamical Systems}
We begin this section by some definitions which is an extension to fuzzy topology of the concepts in the previous section and will be used in Theorems \ref{t1} to \ref{t8} and Section 4.  \\Suppose $X$ is a nonempty set and $f:X\rightarrow X$ is an arbitrary map. The $n-$th iteration $f^{n}$  of the map $f$ is defined inductively by $f^{1} = f$
 and $f^{n}= f o f ^{n-1}$ for $ n \geq2$  and the orbit of $x$ under $f$ is the set of points $\{x, f(x), f^{2}(x), . . .\}$. It is well known\cite{Z} that $(X, f)$  induces a dynamical
system $(\mathbb{F}(X), \widehat{f})$, called fuzzy discrete dynamical system, that is defined on the space
$\mathbb{F}(X)$ of all fuzzy compact subsets of $X$. The map
$\widehat{f}$ is usually called fuzzification (see \cite{Ka} and \cite{Kb} for more details. Now, we consider a  fuzzy discrete dynamical system, denoted by $(X, M,\ast, f)$,
 as a fuzzy continuous map $f:X\rightarrow X$ with all iterations and identity map, where    $(X, M,\ast)$ is a fuzzy metric space.
%triple $(X,T,\varphi)$, where $T$ is a topological semigroup, $X$ at least is Hausdorff and $$\varphi:T\times X\rightarrow X $$$$(s,x)\rightarrow sx$$ is a continuous
%action on $X$. Thus $s_{1}(s_{2}x)=(s_{1}s_{2})x$ holds for each $(s_{1}, s_{2}, x)$ in $S\times S\times X$
%Note that if $T=\{f^{n}: n=1,2,...\}$ and $f:X\rightarrow X$ is continuous, then $(X, T)$ is a discreet dynamical system in $X$.
\\
\begin{df}Given a number $\delta>0$. A sequence $\{x_{i}\}_{i=0}^{b}$ of points in $X$ is called a \emph{$\delta-F$-pseudo-orbit} of $f$ if there exists \\$T_{0}\in (0,\infty)$ such that $M(f(x_{i}), x_{i+1}, T_{0})>1-\delta$ for all $0\leq i< b$. \end{df}
\indent If $b<\infty$, then we say that the finite $\delta-F$-pseudo-orbit $\{x_{i}\}_{i=0}^{b}$ of $f$ is a \emph{$\delta-F-$chain} of $f$ from $x_{0}$ to $x_{b}.$\\ We say that a fuzzy dynamical system $(X, M , \ast, f)$ has the \\\emph{$F-$shadowing property}, if for any $\epsilon>0$ there is a $\delta>0$ such that every $\delta-F-$pseudo-orbit  $\{x_{i}\}_{i=0}^{b}$ of $f$ can be $\epsilon-F$-traced by some point in $X$ that is:\\
 There is $T_{0}>0$ and $x\in X$ such that $M(f^{i}(x), x_{i}, ~T_{0})>1-\epsilon$, for each $0\leq i\leq b $.
\\
 Let $f$ be a fuzzy continuous map. For any two fuzzy open subsets $U$ and $V$ of $X$, let $$N^{f}(U, V)=\{m\in\mathbb{N}; f^{m}(U)\cap V\neq\emptyset\}.$$ We say that $f$ is \emph{$F$-topological transitive} if for any two fuzzy open subsets $U$ and $V$ of $X$,  $N^{f}(U,~V)\neq \emptyset$. \emph{$F-$Topological mixing}
 means that for any two fuzzy open subsets $U$ and $V$, the set $N^{f}(U, V)$ contains any natural number $n\geq n_{0},$ for some fixed $n_{0}\in \mathbb{N}.$
%\begin{df} The fuzzy dynamical system $(X,f)$ has the $F-$average-shadowing property if for every $\epsilon>0$, there is $\delta>0$ such that for every
%$F-\delta-$pseudo-orbit $\{x_{i}\}_{i=0}^{\infty}$, there exists $y\in X$ and  $T_{1}\in [0,\infty)$ satisfying $M(f^{i}(y),x_{i}, t)>1-\epsilon$ for all
% $i\geq 0$ and $t\geq T_{1}$.\end{df}
From here the section we assume that $(X, M, \ast)$ is a complete and compact fuzzy metric space, $\mathbb{R}^{+}=(0, \infty)$ and $\mathbb{N}$ the positive integers. For brevity, we write the fuzzy metric spaces as $X$,
wherever there is no risk of confusion.\begin{df}A mapping $f$ is called \emph{$F-$chain transitive} if for any two points $x, y\in X$ and any $\delta>0$ there exists a $\delta-F-$chain from $x$ to $y$. The mapping $f$ is called \emph{$F$-chain mixing} if for any two points $x, y \in X$ and any $\delta>0,$ there is a positive integer $N_{0}$ such that for any integer $n \geq N_{0}$ there is a $\delta-F-$chain from $x$ to $y$ of length $n$. \end{df}
\indent Given a sequence $\mu=\{x_{i}\}_{i=0}^{\infty}$, let $$N^{f}(\mu, \delta,~t)=\{i; M(f(x_{i}), x_{i+1}, t)\leq 1-\delta\}$$ and $$N_{n}^{f}(\mu, \delta, t)=N^{f}(\mu, \delta,t)\cap\{1, ..., n-1\}.$$ For a sequence $\mu$ and a point $x$ of $X$, let $$N_{s}^{f}(\mu, x, \delta, t)=\{i; M(f^{i}(x), x_{i},  t)\leq 1-\delta\}$$ and $$N_{sn}^{f}(\mu, x, \delta, t)=N_{s}^{f}(\mu, x, \delta, t)\cap\{1, ..., n-1\}.$$\begin{df} A sequence $\mu$ is a \emph{$\delta-F$-ergodic pseudo orbit}, if there exist $T_{0}\in(0, \infty)$ such that $$\lim_{n\rightarrow\infty}\frac{card( N_{n}^{f}(\mu, \delta, T_{0}))}{n}=0.$$ A $\delta-F$-ergodic pseudo orbit, is said to be \emph{$\epsilon-F$-ergodic shadowed} by a point $x$ in $X$ if $$\lim_{n\rightarrow\infty}\frac{card( N_{sn}^{f}(\mu, x, \delta,~T_{0}))}{n}=0.$$ \end{df}
 A mapping $f$ has \emph{$F$-ergodic shadowing property} if for any $\epsilon>0$ there is a $\delta>0$ such that any $\delta-F$-ergodic pseudo orbit of $f$ can be $\epsilon-F-$ergodic shadowed by some point in $X.$\\
 Now we are ready to show the results of this paper.
 \begin{them}\label{t1}Any mapping with $F-$ergodic shadowing property is $F-$chain transitive.\end{them}
 \begin{proof}
Let $f$ be a mapping that has the $F-$ergodic shadowing property and $x,~y \in X$. Suppose $\epsilon$ is an arbitrary positive number and $\delta>0$ is a number corresponding to $\epsilon$ by the $F-$ergodic shadowing property.\\ Let $\{[a_{i}, b_{i}]; i=0, 1, 2, ...\}$ be a family of intervals such that $a_{0}=0$, $b_{0}=1$ and $a_{k}=b_{k-1}+k$, $b_{k}=a_{k}+k+1$ for all $k\geq1$. This is clear that the set $\{a_{i},~b_{i};  i=0,~1,~2,~...\}$ has density zero in natural numbers. Define the sequence $\{x_{i}\}_{i=0}^{\infty}$ by the following formulas:\[ x_{i} = \left\lbrace
  \begin{array}{c l}
   f^{i-a_{k}}(x)& \text{if ~$~ a_{k}\leq i < b_{k}$},\\
    f^{i-b_{k}}(x_{k}) & \text{if ~$~b_{k}\leq i < a_{k+1}$}.
  \end{array}
\right. \]
 \indent By definition of $\{x_{i}\}_{i=0}^{\infty}$, we have $M(f(x_{i}),  x_{i+1},t)=1$ where $i=a_{k}$ or $i=b_{k}$ for some $k\geq0$. This implies that $N^{f}(\mu, \delta, t)$ is a subset of $\{a_{k}, b_{k}; k=0, 1, 2, ....\}$. Thus\\ $\lim_{n\rightarrow\infty}\frac{N^{f}(\mu, \delta, t)}{n}$$$\leq\lim_{n\rightarrow\infty}\frac{card\{a_{k}, b_{k}; k=0, 1, 2, ....\}\cap \{0, 1, 2, ..., n-1\}}{n}=0.$$ Where $card\{a_{k}, b_{k}; k=0, 1, 2, ....\}$ is cardinality of \\$\{a_{k}, b_{k}; k=0, 1, 2, ....\}.$ Then $\mu=\{x_{i}\}_{i=0}^{\infty}$ is a $\delta-F-$ergodic pseudo orbit and can be $\epsilon-F-$ergodic shadowed by some point $z$ in $X$. Since $N_{s}^{f}(\mu, ~z,~ \epsilon,~t)$ has density zero, for some $t\in(0,~\infty).$ Then we can find $i, j, l, s\in \mathbb{N}$ with $j<l$ such that $M(f^{j}(x), f^{i}(z),  t)\leq 1-\epsilon$, $M(f^{l}(z), w,  t)\leq1-\epsilon$ and $w\in f^{-s}(y)$.\\
 \indent Therefore $$\{x, f(x), ..., f^{j-1}(x), f^{i}(z), ..., f^{l-1}(z), w, f(w), ..., f^{s-1}(w), y\}$$ is an $\epsilon-F-$pseudo orbit from $x$ to $y.$ \end{proof}\begin{them}\label{t2}If $f$ has the $F-$ergodic shadowing property then for any natural number $k$ the mapping $f^{k}$ has also the $F-$ergodic shadowing property.\end{them}\begin{proof}Let $f$ has the $F-$ergodic shadowing property. Given $\epsilon>0$, let $\delta>0$ be an $\epsilon$ modulus of $F-$ergodic shadowing property of $f$. Suppose  $\{x_{i}\}_{i=0}^{\infty}$ is an $\delta-F-$ergodic pseudo orbit for $f^{k}$. So there exist $T_{0}\in(0,\infty)$ such that $\lim_{n\rightarrow\infty}\frac{\mid N_{n}^{f^{k}}(\{x_{i}\}_{i=0}^{\infty}, x, \delta, T_{0})\mid}{n}=0.$
 We define the sequence $\{z_{j}\}_{j=0}^{\infty}$ with\\ $z_{ki+l}=f^{l}(x_{i})$ for $0\leq l< k, n\in \mathbb{Z_{+}}$, that is $$\{z_{j}\}_{j=0}^{\infty}=\{x_{0}, f(x_{0}),~...,~f^{k-1}(x_{0}),~x_{1}, f(x_{1}), ..., f^{k-1}(x_{1}),...\}.$$ It is easy to see that:
 \[ M(f(z_{ki+l}), z_{ki+l+1},  T_{0}) = \left\lbrace
  \begin{array}{c l}
 0& \text{if $ 0\leq l< k-1$},\\
    M(f^{k}(x_{i}), x_{i+1}, T_{0}) & \text{if ~$~l=k-1$}.
 \end{array}
  \right. \]\indent So $\{z_{j}\}_{j=0}^{\infty}$ is an $\delta-F-$ergodic pseudo orbit, $f$. Then there exists $x\in X$ such that $\{z_{j}\}_{j=0}^{\infty}$ is $\epsilon-F-$shadowed by $x$ in dynamical system $(X,  M, \ast,  f).$ Since $M(f^{ki}(x), z_{ki},  T_{0})=M((f^{k})^{i}(x), x_{i},  T_{0})$, so $$card( N_{sn}^{f^{k}}(\{x_{i}\}_{i=0}^{\infty}, x, \delta,T_{0}))\leq card( N_{skn}^{f}(\{z_{j}\}_{j=0}^{\infty}, x, \delta,T_{0})).$$ This implies that
  \begin{eqnarray*}
  \lim_{n\rightarrow\infty}\frac{card( N_{sn}^{f^{k}}(\{x_{i}\}_{i=0}^{\infty},x,\delta,T_{0}))}{n}\\ \nonumber
  &\leq k\lim_{n\rightarrow\infty}\frac{card( N_{skn}^{f}(\{z_{j}\}_{j=0}^{\infty},x,\delta,T_{0}))}{kn}=0.
  \end{eqnarray*}\end{proof}
  \begin{them}\label{t3}Let $f$ be a $F-$continuous onto map. For the fuzzy dynamical system $(X, ~M,~\ast, ~f)$, if $f$ has the $F-$ergodic shadowing property then it has the $F-$shadowing property.\end{them}\begin{proof} First, we prove that any finite $F$-pseudo orbit of $f$ can be shadowed by true orbit. Let $\mu=\{z_{j}\}_{j=0}^{n}$ be a finite $\delta-F-$pseudo orbit, by Theorem 3.5, we can find a $\delta-F-$pseudo orbit \\$\lambda=\{z_{n}=y_{0},~...,~y_{m}=z_{0}\}$ from $z_{n}$ to $z_{0}$. Then \\$\omega=\{z_{0}, ..., z_{n}, y_{1}, ..., y_{m-1}, z_{0}, z_{1}, ..., z_{n}, y_{1},~...\}$ is a $\delta-F-$ergodic pseudo orbit and then can be $F-$ergodic shadowed by some point $x\in X.$ If  the set $N_{s}^{f}(\omega, x, \epsilon, T_{0})$ meet every $\mu$ interval then it would have positive density that is a contradict. Hence there exist at least one $k\in \mathbb{N}$ such that  $M(f^{k+i}(x), z_{i},t)>1-\epsilon$, for every $0\leq i\leq n$. In other words $\mu$ is $\epsilon-F-$ shadowed by $f^{k}(x)$.\\
  \indent Now suppose $\{z_{j}\}_{j=0}^{\infty}$ is an $\delta-F-$pseudo orbit of $f$, for\\ $k>0$, denote $x_{i}=z_{i}~(0\leq i\leq k)$. Then $\{x_{i}\}_{i=0}^{k}$ is an $\delta-F-$pseudo orbit of $f$, and so is $\epsilon-F-$shadowed by some point $w_{k}\in X$. Let $z\in X$ be an accumulation point of $\{w_{k}\}_{k=1}^{\infty}$, because $X$ is a compact and complete fuzzy metric space. We show that $\{z_{j}\}_{j=0}^{\infty}$ is $\epsilon-F-$shadowed by $z$.\\ \indent Fix $n\geq0$ and $t>0$. By Lemma \ref{l1}, there exist $\epsilon^{'}>0$ such that $(1-\epsilon^{'})(1-\epsilon^{'})\geq 1-\epsilon.$  Since $f^{n}$ is a fuzzy continuous map then there exist $\delta^{'}>0$ and $t^{'}>0$ such that if $M(z, x, t^{'})>1-\delta^{'}$ then $M(f^{n}(z), f^{n}(x), t)>1-\epsilon^{'}$, for every $x\in X.$\\ Take sufficiently large $k_{n}$ satisfying $k_{n}>n$ and \\$M(z, w_{k_{n}}, t^{'})>1-\delta^{'}$. Then
  \begin{eqnarray*}
  M(f^{n}(z), z_{n}, 2t)\\
  &\geq& M(f^{n}(z), f^{n}(w_{k_{n}}), t)\ast M(f^{n}(w_{k_{n}}), z_{n}, t)\\
  &\geq&(1-\epsilon^{'})\ast(1-\epsilon^{'})\\
  &\geq& 1-\epsilon.\end{eqnarray*} This prove that $\{z_{j}\}_{j=0}^{\infty}$ is $2\epsilon-$shadowed by $z.$  \end{proof} \begin{them}\label{t8}If $f$ has the $F$-shadowing and $F$-chain mixing properties then it has $F-$topologically mixing property.\end{them}\begin{proof}Suppose that $f$ has the $F$-chain mixing property. Given two fuzzy open subsets $U$ and $V$ of $X$. Choose $x\in~ U$, $y\in V$, $t>0$ and $\epsilon>0$ such that $B(x, \epsilon, t)\subset U$ and $B(y, \epsilon, t)\subset V.$ Let $\delta>0$ be an $\epsilon-$modulus of $F-$shadowing property. Since $f$ is $F-$chain mixing then there is a positive integer $N_{0}$ such that for any integer $n\geq N_{0},$ there is a $\delta-F-$chain, $x=x_{1}, x_{2},...., x_{n}=y$, of length $n$ from $x$ to $y$. By $F$-shadowing property we can find $z\in X$ such that $M(f^{i}(z), x_{i}, t)>1-\epsilon,$ for all $1\leq i\leq n$. Particulary $M(z, x, t)>1-\epsilon$ and \\$M(f^{n}(z), y, t)>1-\epsilon$. So $z\in U$ and $f^{n}(z)\in V.$ Then \\$f^{n}(U)\cap V\neq\emptyset$. \end{proof}
  \section{Examples} In the following, we give an interval map that has the $F-$topological mixing property.
  \begin{exam}\label{e6} Let $X=[0, 1]$ and $a\ast b=ab$ for all $a, b \in X$. Let $(X, M_{d},\ast)$ be the standard fuzzy metric space induced by $d$, where $d(x,y)=|x-y|$ for all $x, y\in X$. Suppose $f:X\longrightarrow X$ is the tent map which is defined by:
 \[ f(x)= \left\lbrace
  \begin{array}{c l}
 2x& \text{if ~$~ 0\leq x\leq \frac{1}{2}$},\\
    2(1-x) & \text{if ~$\frac{1}{2}\leq x\leq 1$}.
 \end{array}
  \right. \] It is known that the tent map has the shadowing and chain mixing properties (see \cite{CKY,RW}) and we show that $f$ has the $F-$shadowing and $F-$chain mixing properties.\\ Let $\epsilon>0$.
 It is clear that there is $T_{0}>0$, such that $M_{d}(x, y, T_{0})>1-\epsilon$ for all $x, y\in X$. So  $f$ has the $F$-shadowing and $F$-chain-mixing properties.
    It follows from Theorem \ref{t8} that $f$ has the $F-$topological mixing property.\end{exam}\begin{rem}\cite{CKY}\label{rem2}
    Consider the family of tent maps, i.e., the piecewise linear maps $f_{\beta}(x):[0,1]\rightarrow [0,1]$, $\sqrt{2}\leq\beta\leq 2$, defined by:
     \[ f_{\beta}(x)= \left\lbrace
  \begin{array}{c l}
 \beta x& \text{if ~$~ 0\leq x\leq \frac{1}{2}$},\\
    \beta(1-x) & \text{if ~$\frac{1}{2}\leq x\leq 1$}.
 \end{array}
  \right. \]
 In \cite{CKY}, the authors proved that,
tent maps have the shadowing property  for almost all parameters $\beta$, although they fail to
have the shadowing property for an uncountable, dense set of parameters. But, similar to Example \ref{e6}, for an arbitrary positive number $\epsilon$, $M_{d}(x, y, T_{0})>1-\epsilon$ for all $x, y\in X$. So $(X, ~M_{d}, ~\ast, ~f_{\beta})$ has the $F-$shadowing property, for all $\sqrt{2}\leq\beta\leq 2$. \end{rem} From Example \ref{e6} it can be seen that for metric $M_{d}$ any continuous map has the $F$-shadowing property. Since the topology $\tau_{d}$ generated by $d$ coincides with the
topology  $\tau_{M_{d}}$ generated by the fuzzy metric $M_{d}$, so the metric $M_{d}$ is not a help for studying the dynamics.\\ \indent The following example shows that the \textbf{$F-$shadowing property} is not completely similar to the shadowing property, although fuzzy metric space $(X, M ,\ast)$ define the usual topological structure on $X$.
 \begin{exam}\label{e7} Let $X=(0,~1]$ and $f: X\rightarrow X$ be the following function (see Figure $1$):
 \[ f(x)= \left\lbrace
  \begin{array}{ccc}
\frac{3}{4}x+\frac{1}{8} &\text{if ~$~ 0< x\leq~ \frac{1}{2}$},\\
\frac{3}{2}x-\frac{1}{4} & \text{if ~$\frac{1}{2}\leq x\leq~ \frac{3}{4}$},\\
\frac{1}{2}x+\frac{1}{2} & \text{if ~$\frac{3}{4}\leq x\leq~ 1$}.
 \end{array}
  \right. \]
  This kind of discrete dynamical systems is an important example to investigate various shadowing properties; see  [8 Example $5.4.$] and [11, Example $4.2.$]
  for more details.\\
 %with the property: $f(x)>x$ if and only if $x\in (0,\frac{1}{2})\cup (\frac{1}{2},1]$ and $f(x)>\frac{1}{8}$ for all $x\in X$.
\indent(a) It is well known that $f$ does not have the shadowing property \cite{CL}.\\ \indent(b) Suppose $(X,~M_{d}, \ast)$ is the standard  fuzzy metric space in Example \ref{e1}. It is clear that $f$ is a fuzzy continuous function, respect to the fuzzy metric $M_{d}$. \\ \indent Since $|x-y|<1$ for all $x, y\in X$. Thus, given $\epsilon>0$ there exist $T_{0}>0$ such that $M_{d}(x, y, t)>1-\epsilon$ for all $x,~y\in X$ and $t>T_{0}$. So this is clear that for every $\epsilon-F$-pseudo orbit, $\{x_{i}\}_{i=0}^{\infty}$, $M_{d}(f^{i}(\frac{1}{2}), x_{i}, t)>1-\epsilon$ for all $i\geq 0$ and $t>T_{0}$. Therefore $(X, M_{d}, \ast, f)$ has the $F-$shadowing property.\\
 \indent(c) Now we show that $(X, M, \ast,  f)$ does not have the $F-$shadowing property, where $(X, M, \ast)$ is fuzzy metric space in Example \ref{e2}. At first we show \\ that $f$ is a fuzzy continuous function respect to fuzzy \\metric $M$. \\  Assume that $\epsilon$ and $t$ are arbitrary positive numbers and $\varphi:\mathbb{R}^{+}\rightarrow (0, 1]$ is the map in Example \ref{e2}.
 \begin{figure}\label{f1}
\begin{center}
\includegraphics[width=8 cm , height= 8cm]{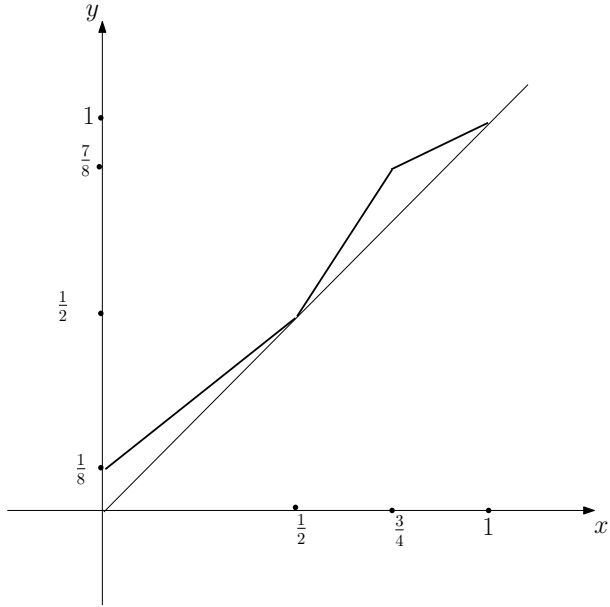}\\
\end{center}
\begin{center} \caption{The graph of the function $f$ from Example \ref{e7}}
\end{center}
\end{figure}
  By investigating various conditions of $x$ and $y$, we can prove that  $$\frac{min\{f(x), f(y)\}}{max\{f(x), f(y)\}}>\frac{1}{10}\frac{min\{x, y\}}{max\{x, y\}},$$
   for all $x, y\in X$. For example if $x,y\in[0,\frac{1}{2}]$ and $x<y$ then $f(x)=\frac{3}{4}x+\frac{1}{8}$ and $f(y)=\frac{3}{4}y+\frac{1}{8}$. So
   $\frac{min\{f(x), f(y)\}}=\frac{3}{4}x+\frac{1}{8}$ and ${max\{f(x), f(y)\}}=\frac{3}{4}y+\frac{1}{8}$
   and clearly $\frac{\frac{3}{4}x+\frac{1}{8}}{\frac{3}{4}y+\frac{1}{8}}>\frac{1}{10}\frac{x}{y}$.\\
 %Let Since $f$ is a continuous function, there is $0<\delta<1$ such that if $|x-y|<\delta$ then $|f(x)-f(y)|<\frac{1}{8}\epsilon$, for all $x,y \in X$.
  This implies that $M(f(x), f(y), t)>M(x, y, \frac{t}{10})$. So this is clear that if $M(x,y,\frac{t}{10})>1-\epsilon$ then  $M(f(x), f(y), t)>1-\epsilon$.
   This prove that $f$ is a fuzzy continuous function on $(X, M, \ast)$.
 %Suppose $\delta^{'}=1-\varphi(t)(1-\delta)$ and $M(x,y,t)>1-\delta^{'}$, without loss of generality we assume that $x<y$. Then $\frac{x}{y}\varphi(t)>1-\delta^{'}$,
 %so $\frac{x}{y}\varphi(t)>1-\delta^{'}$$y-x<y\delta^{'}\leq\delta$. This implies that $f(y)-f(x)<\frac{1}{8}\frac{\epsilon}{\varphi(t)}<f(y)\frac{\epsilon}{\varphi(t)}$.
 %Therefore $M(x,y,t)= \frac{f(x)}{f(y)}.\varphi(t)>1-\epsilon$, because of the monotonicity of $f$.
 \\
  Let $\epsilon=\frac{1}{5}$. Since $f(0)=\frac{1}{8}, f(\frac{1}{2}=\frac{1}{2} and f(\frac{3}{4}=\frac{7}{8}$, for any $\delta>0$ we can find the points $a, b\in X$ such that $a<\frac{1}{2}<b$, $\frac{min\{f(a), \frac{1}{2}\}}{max\{f(a), \frac{1}{2}\}}>1-\delta$ and $\frac{min\{f(\frac{1}{2}),  b\}}{max\{f(\frac{1}{2}), b\}}>1-\delta$. Up to this point, there is a $\delta-F-$pseudo orbit that contains  $\frac{1}{4}$ and $1$. Let $x\in X$, if $x\leq \frac{1}{2}$ then $M(f^{i}(x), \frac{1}{4}, t)<1-\frac{1}{5}$, for all $t>1 $ and for all $i\geq 0$. And, if $x<\frac{1}{2}$ then $M(f^{i}(x), 1, t)<1-\frac{1}{5}$, for all $t>1 $ and for all $i\geq 0$. This implies that $(X, M, \ast, f)$ does not have the $F-$shadowing property.
\\
 \indent(d) Suppose $(X, M_{1}, \ast)$ is the fuzzy metric space in Example \ref{e3}. Given $\epsilon>0$, since $f$ is continuous then there is $\delta>0$ such that if $|y-x|<\delta $ then $|f(y)-f(x)|<\frac{\epsilon}{8}$. Let $x<y$ and $\frac{x}{y}>1-\delta$. This implies that $y-x<\delta$, so $f(y)-f(x)<\frac{\epsilon}{8}$ and so \\$\frac{f(x)}{f(y)}>1-\frac{\epsilon}{8f(y)}>1-\epsilon$. Thus  $f$ is a fuzzy continuous function on $(X, M_{1}, \ast)$. By similar argument to the previous paragraph, one can prove that $( X, M_{1}, \ast, f)$ does not have the $F-$shadowing property.\end{exam}
  \begin{figure}\label{f9}
\begin{center}
\includegraphics[width=8 cm , height= 8cm]{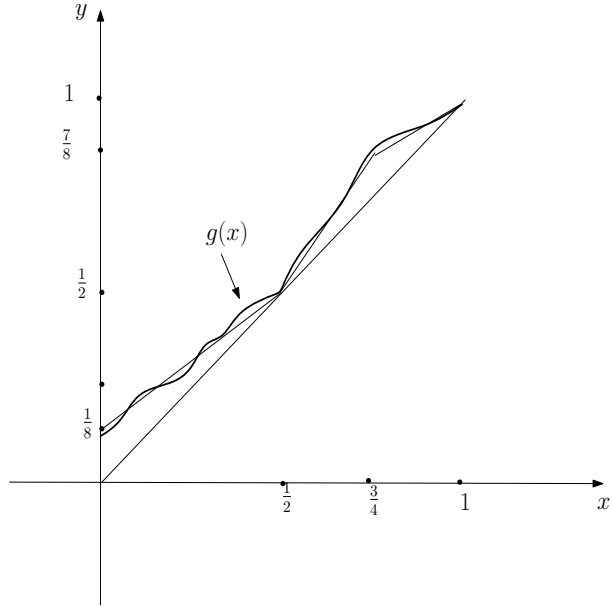}\\
\end{center}
\begin{center} \caption{The graph of the functions $f$ and $g$ from Example \ref{e7} that are near together.}
\end{center}
\end{figure}
It is a remarkable fact that, one can prove analogous results for any map $g$ sufficiently close to $f$.
% that $g\frac{1}{2}=\frac{1}{2}$ and $g(1)=1$ .
\begin{exam}
Suppose $\alpha<\frac{1}{128}$ is a positive number and \\$g:[0,1]\rightarrow [0,1]$ be a continuous monotone map such that $g(x)>x$ if and only if $x\in[0,\frac{1}{2})\cup(\frac{1}{2},1)$,  $g(1)=1$, $g(\frac{1}{2})=\frac{1}{2}$ and \\$sup\{|f(x)-g(x)|: x\in [0,1]\}<\alpha$ (see Figure $2$ ). At first we show that $( X,M_{1}, \ast, g)$ is a fuzzy continuous function. Since $( X,M_{1}, \ast, g)$ is a fuzzy continuous function, this is sufficient to prove that $M_{1}(g(x),g(y),t)>\frac{1}{2} M_{1}(f(x),f(y),t)$ for all $x,y \in X$ and $t\in [0,1]$. \\Let $x, y \in [0,1]$ which $x<y$. It is clear that $f(x)f(y)>\alpha(f(x)+f(y))$. So $\frac{f(x)-\alpha}{f(y)+\alpha}>\frac{1}{2}\frac{f(x)}{f(y)}$. This implies that $\frac{g(x)}{g(y)}>\frac{f(x)-\alpha}{f(y)+\alpha}>\frac{1}{2}\frac{f(x)}{f(y)}$. Then $$\frac{min\{g(x),  g(y)\}}{max\{g(x),g(y)\}}=\frac{g(x)}{g(y)}>\frac{1}{2}\frac{f(x)}{f(y)}=\frac{1}{2}\frac{min\{f(x),  f(y)\}}{max\{f(x),f(y)\}}.$$So $M_{1}(g(x),g(y),t)>\frac{1}{2} M_{1}(f(x),f(y),t)$ for all $x,y \in X$ and $t\in [0,1]$.   \\Similar argument to $f$ in   Example \ref{e7} shows that the map $g$ does not have the shadowing property and $( X,M_{1}, \ast, g)$ does not have the $F-$shadowing property.
\end{exam}
\section{Conclusion}In this paper, we extend some notions and results known  in discrete dynamical systems  on fuzzy discrete dynamical systems and some related properties are investigated. In particular, the relation between various fuzzy shadowing property and fuzzy mixing is investigated. More precisely, the following implications are obtained:\\
$(a)$ $F$-ergodic shadowing property implies  $F-$chain transitivity.\\
 %$(b)$ If $f$ has the $F-$ergodic shadowing property then for any natural number $k$ the mapping $f^{k}$ has also the $F-$ergodic shadowing property.\\
 $(b)$ $F-$ergodic shadowing property implies $F-$shadowing property.\\
 $(c)$ $F$-shadowing and $F$-chain mixing properties imply $F-$topologically mixing property.\\
By definitions, most of the proofs were straight forward, but because of Example \ref{e7} we expect that a theorem similar to Theorem \ref{gha} in fuzzy discrete dynamical systems can be obtained for special fuzzy metric spaces. In particular, We are going to study and extend the other types of shadowing and their equivalence to fuzzy topology in another papers. Investigating the relation between chaos and stability with shadowing in fuzzy topology will become our future research topic. We finish this paper with the following  important questions: \\$1)$ For a fuzzy discrete dynamical system $f:X\longrightarrow X$, are the  $F-$ergodic shadowing and $F-$topological mixing properties equivalent?\\$2)$ Does the fuzzy discrete dynamical system $(X, M, \ast, f_{\beta})$ have the $F$-shadowing and $F$-chain mixing properties, when $f_{\beta}$ is the tent map and $(X, M_{1}, \ast)$ is the fuzzy metric space presented in  Example \ref{e2} or Example \ref{e3}?\\
\textbf{}

%-----------------------------------------------------------------------------
%-----------------------------------------------------------------------------

\bigskip
\bigskip

%{\bf Received: Month xx, 200x}

%{\footnotesize \pn{\bf Mehdi Fatehi Nia}\;

%{\footnotesize \pn{\bf Second Author}\; \\ {Department of
%Mathematics}, {University
%of ABCD, P.O.Box xxxx,} {City, Country}\\
%{\tt Email: xxx\@xxxxx}\\
\end{document}